\documentclass[12pt]{amsart}
\usepackage{amsfonts, amssymb}
\usepackage{amsmath}
\usepackage{amsthm}
\usepackage{cite}
\numberwithin{equation}{section}

\newtheorem{theorem}{Theorem}[section]
\newtheorem{lemma}[theorem]{Lemma}

\begin{document}

\title[Malmquist type difference equations]{On meromorphic solutions of Malmquist type difference equations}

\author{Yueyang Zhang}
\address{School of Mathematics and Physics, University of Science and Technology Beijing, No.~30 Xueyuan Road, Haidian, Beijing, 100083, P.R. China}
\email{zhangyueyang@ustb.edu.cn}
\thanks{The first author is supported by a Project funded by China Postdoctoral Science Foundation~{(2020M680334)} and the Fundamental Research Funds for the Central Universities~{(FRF-TP-19-055A1)}. The second author is supported by a Visiting Professorship for Senior Foreign Experts by Ministry of Science and Technology of the People's Republic of China~{(G2021105019L)}}

\author{Risto Korhonen}
\address{Department of Physics and Mathematics, University of Eastern Finland, P.O. Box 111, FI-80101 Joensuu, Finland,\newline
and\newline
\indent School of Mathematics and Physics, University of Science and Technology Beijing, No.~30 Xueyuan Road, Haidian, Beijing, 100083, P.R. China
}
\email{risto.korhonen@uef.fi}

\subjclass[2010]{Primary 39A10; Secondary 30D35 \and 39A12}

\keywords{Malmquist type difference equations; Nevanlinna theory; Meromorphic solutions; Differential equations; Continuum limit}

\date{\today}

\commby{}

\begin{abstract}
Recently, the present authors used Nevanlinna theory to provide a classification for the Malmquist type difference equations of the form $f(z+1)^n=R(z,f)$ $(\dag)$ that have  transcendental meromorphic solutions, where $R(z,f)$ is rational in both arguments. In this paper, we first complete the classification for the case $\deg_{f}(R(z,f))=n$ of~$(\dag)$ by identifying a new equation that was left out in our previous work. We will actually derive all the equations in this case based on some new observations on~$(\dag)$. Then, we study the relations between $(\dag)$ and its differential counterpart $(f')^n=R(z,f)$. We show that most autonomous equations, singled out from~$(\dag)$ with $n=2$, have a natural continuum limit to either the differential Riccati equation $f'=a+f^2$ or the differential equation $(f')^2=a(f^2-\tau_1^2)(f^2-\tau_2^2)$, where $a\not=0$ and $\tau_i$ are constants such that $\tau_1^2\not=\tau_2^2$. The latter second degree differential equation and the symmetric QRT map are derived from each other using the bilinear method and the continuum limit method.
\end{abstract}

\maketitle


\section{Introduction}\label{intro} 

The classical Malmquist theorem~\cite{malmquist1913fonctions} states that: If the first order differential equation $f'=R(z,f)$, where $R(z,f)$ is rational in both arguments, has a transcendental meromorphic solution, then this equation reduces into the Riccati equation
    \begin{equation}\label{riccati}
    f'=a_2 f^2 + a_1 f + a_0,
    \end{equation}
where $a_0$, $a_1$ and $a_2$ are rational functions. Generalizations of Malmquist's theorem for the equation
    \begin{equation}\label{first_order_di_n}
    (f')^n=R(z,f),\qquad n\in\mathbb{N},
    \end{equation}
have been given by Yosida \cite{yosida:33} and Laine \cite{laine:71}. Steinmetz \cite{steinmetz:78}, and Bank and Kaufman \cite{bank1980growth} proved that if \eqref{first_order_di_n} has rational coefficients and a transcendental meromorphic solution, then by a suitable M\"obius transformation, \eqref{first_order_di_n} can be either mapped to the Riccati equation \eqref{riccati}, or to one of the equations in the following list:
    \begin{eqnarray}
    (f')^2 &=& a(f-b)^2(f-\tau_1)(f-\tau_2),\label{MYS1}\\
    (f')^2 &=& a(f-\tau_1)(f-\tau_2)(f-\tau_3)(f-\tau_4),\label{MYS2}\\
    (f')^3 &=& a(f-\tau_1)^2(f-\tau_2)^2(f-\tau_3)^2,\label{MYS3}\\
    (f')^4 &=& a(f-\tau_1)^2(f-\tau_2)^3(f-\tau_3)^3,\label{MYS4}\\
    (f')^6 &=& a(f-\tau_1)^3(f-\tau_2)^4(f-\tau_3)^5,\label{MYS5}
    \end{eqnarray}
where $a$ and $b$ are rational functions, and $\tau_1,\ldots,\tau_4$ are distinct constants.
See \cite[Chapter~10]{Laine1993} for more information about Malmquist--Yosida--Steinmetz type theorems.

Recently, the present authors \cite{Korhonenzhang2020,zhangkorhonen2022} used Nevanlinna theory to provide a  classification for a natural difference analogue of equation \eqref{first_order_di_n}, i.e., the first-order difference equation
\begin{equation}\label{first_order_de_n}
f(z+1)^n=R(z,f),
\end{equation}
where $n\in \mathbb{N}$ and $R(z,f)$ is rational in both arguments. In particular, it is shown in \cite{Korhonenzhang2020} that if the difference equation \eqref{first_order_de_n} has a transcendental meromorphic solution $f$ of hyper-order $<1$, then either $f$ satisfies a difference linear or Riccati equation
    \begin{eqnarray}
    f(z+1)&=& a_1(z)f(z)+a_2(z),\label{lineareq}\\
    f(z+1)&=&\frac{b_1(z)f(z)+b_2(z)}{f(z)+b_3(z)}\label{driccati0},
     \end{eqnarray}
where $a_i(z)$ and $b_j(z)$ are rational functions, or, by implementing a transformation $f\rightarrow \alpha f$ or $f\rightarrow 1/(\alpha f)$ with an algebraic function $\alpha$ of degree at most~2, \eqref{first_order_de_n} reduces into one of the following equations:
   \begin{eqnarray}
   f(z+1)^2 &=& 1-f(z)^2,\label{yanagiharaeq11 co}\\
   f(z+1)^2 &=& 1-\left(\frac{\delta(z) f(z)-1}{f(z)-\delta(z)}\right)^2,\label{yanagiharaeq12 co}\\
   f(z+1)^2 &=& 1-\left(\frac{f(z)+3}{f(z)-1}\right)^{2},\label{yanagiharaeq13 co}\\
   f(z+1)^2 &=& \frac{f(z)^2-\kappa^2}{f(z)^2-1},\label{yanagiharaeq14 co}\\
   f(z+1)^3 &=& 1-f(z)^{-3},\label{yanagiharaeq15 co}
  \end{eqnarray}
where $\delta(z)\not\equiv\pm1$ is an algebraic function of degree $2$ at most and $\kappa^2\not=0,1$ is a constant. Finite-order meromorphic solutions of the autonomous forms of the equations \eqref{lineareq}--\eqref{yanagiharaeq15 co} are presented explicitly in \cite{Korhonenzhang2020}. These results provide a natural difference analogue of Steinmetz' generalization of Malmquist's theorem in the sense of Ablowitz, Halburd and Herbst \cite{AblowitzHalburdHerbst2000}, who suggested that the existence of sufficiently many finite-order meromorphic solutions of a difference equation is a good candidate for a difference analogue of the Painlev\'e property \cite{halburdrk:LMS2006}. It was shown that the finite-order condition of the proposed difference Painlev\'e property can be relaxed to hyper-order strictly less than one in \cite{halburdkt:14TAMS}, and recently to hyper-order equal to one limitedly in~\cite{Korhonenkazhzh2020,zhengR:18}. Further, by discarding the assumption that the meromorphic solution is of hyper-order $<1$ and considering transcendental meromorphic solutions of \eqref{first_order_de_n} with $\deg_f(R(z,f))=n$, it was shown in \cite{Korhonenzhang2020} that either $f$ satisfies \eqref{lineareq} or \eqref{driccati0}, or \eqref{first_order_de_n} can be transformed into one of the equations \eqref{yanagiharaeq11 co}--\eqref{yanagiharaeq15 co}, or one of the following equations:
   \begin{eqnarray}
   f(z+1)^2 &=& \eta^2(f(z)^2-1), \label{intro_eq_list2_1}\\
   f(z+1)^2 &=& 2(1-f(z)^{-2}), \label{intro_eq_list2_2}\\
   f(z+1)^2 &=& \frac{1+f(z)^2}{1-f(z)^2}, \label{intro_eq_list2_3}\\
   f(z+1)^2 &=& \theta\frac{f(z)^2-\kappa_1f(z)+1}{f(z)^2+\kappa_1f(z)+1}, \label{intro_eq_list2_4}\\
   f(z+1)^3 &=& 1-f(z)^3,\label{intro_eq_list2_5}
  \end{eqnarray}
where $\theta=\pm1$, $\eta\not=1$ is the cubic root of $1$ and $\kappa_1$ is a constant such that $\kappa_1^2(\kappa_1^2-4)=2(1-\theta)\kappa_1^2-8(1+\theta)$. Transcendental meromorphic solutions of the five equations above are elliptic functions composed with entire functions and have hyper order $\geq 1$.

This paper has two purposes. The first one is to complete the classification for the case $\deg_f(R(z,f))=n$ of \eqref{first_order_de_n}. When classifying equation \eqref{first_order_de_n} for the case $\deg_f(R(z,f))\not=n$ in \cite{zhangkorhonen2022}, we made some new observations that also apply to equation \eqref{first_order_de_n} in the case where $\deg_f(R(z,f))=n$. In the next section~\ref{An elementary derivation of the first 12 equations}, we will use these new observations to derive the ten equations \eqref{yanagiharaeq11 co}--\eqref{intro_eq_list2_5} in a straightforward manner and, at the same time, identify the following equation which was omitted in\cite{Korhonenzhang2020}:\begin{equation}\label{intro_eq_list2_6}
    f(z+1)^2 = \frac{1}{2}\frac{(1+\delta)^2}{1+\delta^2}\frac{(f-1)(f-\delta^2)}{(f-\delta)^2},
    \end{equation}
where $\delta\not=0,\pm1,\pm i$ is a constant such that
    \begin{equation}\label{Poiu}
    8\delta^5(\delta^2+1)-(\delta+1)^4=0.
    \end{equation}
Transcendental meromorphic solutions of \eqref{intro_eq_list2_6} are Jacobi elliptic functions composed with entire functions and have hyper-order at least~1, as is shown in section~\ref{An elementary derivation of the first 12 equations} below. With this new equation \eqref{intro_eq_list2_6}, the results in \cite{Korhonenzhang2020,zhangkorhonen2022} can be summarised as: If equation \eqref{first_order_de_n}, where $R(z,f)$ is rational in both arguments, has a transcendental meromorphic solution, then \eqref{first_order_de_n} can be reduced into one out of~$30$ equations. Moreover, the autonomous versions of all these~$30$ equations can be solved in terms of elliptic functions, exponential type functions or functions which are solutions to a certain autonomous first-order difference equation having meromorphic solutions with preassigned asymptotic behavior. We mention that Nakamura and Yanagihara \cite{NakamuraYanagihara1989difference} and Yanagihara~\cite{Yanagihara1989difference} have already classified equation \eqref{first_order_de_n} in the case where $R(z,f)$ is a polynomial in $f$ with constant coefficients.

The second purpose of this paper is to describe relations between meromorphic solutions of \eqref{first_order_de_n} and \eqref{first_order_di_n} when $n=2$. In section~\ref{Relations between Malmquist type differential and difference equations}, we will consider relations between meromorphic solutions of the seven equations \eqref{lineareq}--\eqref{yanagiharaeq14 co} and equations \eqref{riccati}, \eqref{MYS1} and \eqref{MYS2} in the autonomous case. Equations \eqref{yanagiharaeq11 co} and \eqref{yanagiharaeq14 co} are, in fact, special cases of the symmetric Quispel--Roberts--Thompson (QRT) map~\cite{QuispelRobertsThompson1988,QuispelRobertsThompson1989}, which will be introduced in section~\ref{Relations between Malmquist type differential and difference equations} below. The symmetric QRT map and \eqref{MYS1} or \eqref{MYS2} are derived from each other by using the \emph{bilinear method} and the \emph{continuum limit method}. The bilinear method was first used by Hirota~\cite{HirotaRI,HirotaRII,HirotaRIII} to find nonlinear partial difference equations that are difference analogues of some basic partial differential equations. The application of this method here implies that each of the four difference equations \eqref{yanagiharaeq11 co}--\eqref{yanagiharaeq14 co} has a natural continuum limit to equations \eqref{riccati}, \eqref{MYS1} or \eqref{MYS2}. Moreover, in section~\ref{Relations between Malmquist type differential and difference equations}, we also show that each of the five equations \eqref{intro_eq_list2_1}--\eqref{intro_eq_list2_4} and \eqref{intro_eq_list2_6} can be mapped to the symmetric QRT map by doing suitable transformations and thus have a continuum limit to the differential equation \eqref{MYS2}. Finally, in section~\ref{Concluding remarks}, we will provide some comments on our results.

\section{Derivations of the equations \eqref{yanagiharaeq11 co}--\eqref{intro_eq_list2_6}}\label{An elementary derivation of the first 12 equations} 

In this section, we use some new observations on equation \eqref{first_order_de_n} to derive the eleven equations \eqref{yanagiharaeq11 co}--\eqref{intro_eq_list2_6}. We shall assume that the reader is familiar with the standard notation and fundamental results of Nevanlinna theory (see, e.g.,~\cite{Hayman1964Meromorphic}). For a nonconstant meromorphic function $f(z)$, recall that a value $a\in \mathbb{C}\cup\{\infty\}$ is said to be a \emph{completely ramified value} of $f(z)$ when $f(z)-a=0$ has no simple roots. A direct consequence of Nevanlinna's second main theorem is that a transcendental meromorphic function can have at most four completely ramified values in $\mathbb{C}\cup\{\infty\}$.

\subsection{Derivations of the eleven equations \eqref{yanagiharaeq11 co}--\eqref{intro_eq_list2_6}}
To prove the theorems of \cite{Korhonenzhang2020,zhangkorhonen2022}, we have actually used Yamanoi's second main theorem for small functions as targets \cite{yamanoi:04,yamanoi:05}. Denote the field of rational functions by $\mathcal{R}$ and set $\hat{\mathcal{R}}=\mathcal{R}\cup\{\infty\}$. Throughout this section, we say that $c(z)\in \hat{\mathcal{R}}$ is a \emph{completely ramified rational function} of a transcendental meromorphic function $f(z)$ when the equation $f(z)=c(z)$ has at most finitely many simple roots and that $c(z)$ is a \emph{Picard exceptional rational function} of $f(z)$ when $N(r,c,f)=O(\log r)$. We also say that $c(z)$ has multiplicity $m$ if all the roots of $f(z)=c(z)$ have multiplicity at least $m$ with at most finitely many exceptions. Yamanoi's second main theorem yields that a transcendental meromorphic function can have at most two Picard exceptional rational functions and also that the inequality
\begin{equation}\label{multiplicityinequality}
\sum_{i=1}^q\Theta(c_i,f)\leq 2
\end{equation}
holds for any collection of $c_1,\cdots,c_q\in \hat{\mathcal{R}}$ when $f$ is transcendental. Moreover, we have
\begin{theorem}\label{completelyrm}
A non-constant transcendental meromorphic function $f(z)$ can have at most four completely ramified rational functions.
\end{theorem}

All the above statements hold when the field $\mathcal{R}$ is extended slightly to include algebraic functions. For simplicity, in the following we will not distinguish algebraic functions and rational functions. For example, we always use the terms 'completely ramified rational function' and 'Picard exceptional rational function' of $f$ even though sometimes they may actually refer to algebraic functions.

We will restrict ourselves to consider equation \eqref{first_order_de_n} with $n=\deg_f(R(z,f))\geq 2$. Moreover, from now on we use the suppressed notation $\overline{f}=f(z+1)$ and $\underline{f}=f(z-1)$. We write equation \eqref{first_order_de_n} as
\begin{equation}\label{yanagiharaeq0}
\overline{f}^n=R(z,f)= \frac{P(z,f)}{Q(z,f)},
\end{equation}
where $P(z,f)$ and $Q(z,f)$ are two polynomials in $f$ with polynomial coefficients having no common factors and of degrees $p$ and $q$, respectively. Then $\deg_{f} (R(f))=\max\{p,q\}=n$. For simplicity, we may also write $P(z,f)$ and $Q(z,f)$ as
    \begin{equation}\label{P}
    P(z,f) = a_p(f-\alpha_1)^{k_1}\cdots (f-\alpha_\mu)^{k_\mu}
    \end{equation}
and
    \begin{equation}\label{Q}
    Q(z,f) = (f-\beta_1)^{l_1}\cdots (f-\beta_\nu)^{l_\nu},
    \end{equation}
where $a_p$ now denotes a rational function, $\alpha_i$ and $\beta_j$ are in general algebraic functions, distinct from each other, and $k_i,l_j\in \mathbb{N}$ denote the orders of the roots $\alpha_i$ and $\beta_j$, respectively. We may suppose that the \emph{greatest common divisor} of $k_1,\ldots,k_\mu,l_1,\ldots,l_\mu$, denoted by $k=(k_1,\ldots,k_\mu,l_1,\ldots,l_\mu)$, is~$1$. Denote $N_c$ to be the total number of $\alpha_i$ with $k_i<n$ and $\beta_j$ with $l_j<n$. As in \cite{zhangkorhonen2022}, the classification for equation \eqref{yanagiharaeq0} will be according to the number $N_c$ of the roots $\alpha_i$ in \eqref{P} and $\beta_j$ in \eqref{Q} and whether some of them is zero.

First, we summarize the analysis on the roots $\alpha_i$ of $P(z,f)$ and $\beta_j$ of $Q(z,f)$ in the proof of~\cite[Theorem~2]{Korhonenzhang2020} and formulate the following Lemmas~\ref{basiclemma  dis1} and~\ref{basiclemma0  dis2}.

\begin{lemma}\label{basiclemma  dis1}
Let $f$ be a transcendental meromorphic solution of equation \eqref{yanagiharaeq0}. Then $\alpha_i$ is either a Picard exceptional rational function of $f$ or a completely ramified rational function of $f$ with multiplicity $n/(n,k_i)$ and $\beta_j$ is either a Picard exceptional rational function of $f$ or a completely ramified rational function of $f$ with multiplicity $n/(n,l_j)$. Moreover, if $q=0$, then $N_c\in\{2,3\}$; if $q\geq 1$, then $N_c\in\{2,3,4\}$.
\end{lemma}

\renewcommand{\proofname}{Proof.}
\begin{proof}
The first two assertions are direct results from the proof of~\cite{Korhonenzhang2020}. Now the inequality \eqref{multiplicityinequality} implies that $N_c\leq 4$. In particular, when $q=0$, if $N_c=4$, then by the inequality \eqref{multiplicityinequality} we see that the roots $\alpha_1$, $\alpha_2$, $\alpha_3$ and $\alpha_4$ are all completely ramified functions of $f$ with multiplicity~$2$, implying that the order $k_i=n/2$ for all $\alpha_i$, which is impossible. Therefore, when $q=0$ we must have $N_c=2$ or $N_c=3$.
\end{proof}

\begin{lemma}[\cite{Korhonenzhang2020}]\label{basiclemma0  dis2}
Let $f$ be a transcendental meromorphic solution of equation \eqref{yanagiharaeq0}. Then none of $\alpha_i$ in \eqref{P} such that $k_i<n$ is~$0$. Moreover, if $q\geq1$, then after doing a bilinear transformation $f\to 1/f$, if necessary, we may suppose that $p=q=n$.
\end{lemma}

By Lemma~\ref{basiclemma0  dis2}, we may only consider the two cases that $p=n$, $q=0$ or that $p=q=n$ below. Moreover, if $P(z,f)$ has two or more distinct roots, then none of them vanishes identically. We use the idea in the proof of \cite[Lemma~2.3]{zhangkorhonen2022} to prove the following

\begin{lemma}\label{keylemma  dis3}
Let $f$ be a transcendental meromorphic solution of equation \eqref{yanagiharaeq0}. If $\gamma$ is a nonzero rational function, then $\gamma$ cannot be a Picard exceptional rational function of $f$. Moreover, if $\gamma\not\equiv0$ is a completely ramified rational function of $f$ with multiplicity~$m$, then $\omega\gamma$ is a completely ramified rational function of $f$ with multiplicity~$m$, where $\omega$ is the $n$-th root of~1. In particular, if $0$ is a root of $Q(z,f)$ of order less than $n$, then $0$ is not a Picard exceptional rational function of $f$.
\end{lemma}

\renewcommand{\proofname}{Proof.}
\begin{proof}
First, we suppose that $\gamma\not\equiv0$ is a Picard exceptional rational function of $f$. Under our assumptions on equation \eqref{yanagiharaeq0}, we see that at least one of $\alpha_i$ and $\beta_j$ in \eqref{P} and \eqref{Q} is non-zero and of order less than $n$. Denote this $\alpha_i$ or $\beta_j$ by $\delta$ and the order of this root by $t_1$. As in the proof of \cite[Lemma~2.3]{zhangkorhonen2022}, we put
\begin{equation}\label{keylemma equa001  diseq4}
u=\frac{\overline{f}}{f-\delta},  \quad  v=\frac{1}{f-\delta}.
\end{equation}
Then $u$ and $v$ are two functions meromorphic apart from at most finitely branch points and we have
\begin{equation}\label{keylemma equa002  diseq5}
\overline{f}=\frac{u}{v},  \quad  f=\frac{1}{v}+\delta,
\end{equation}
and it follows that \eqref{yanagiharaeq0} becomes
\begin{equation}\label{keylemma equa003  diseq6}
u^n=\frac{P_1(z,v)}{Q_1(z,v)}v^{n_1},
\end{equation}
where $n_1\in \mathbb{Z}$, $P_1(z,v)$ and $Q_1(z,v)$ are two polynomials in $v$ having no common factors and none of the roots of $P_1(z,v)$ or $Q_1(z,v)$ is zero. Denote by $p_1=\deg_{v}(P_1(z,v))$ the degree of $P_1(z,v)$ in $v$ and by $q_1=\deg_{v}(Q_1(z,v))$ the degree of $Q_1(z,v)$ in $v$, respectively. By simple calculations, when $p=n$, $q=0$ we get $n_1=0$, $p_1=p-t_1$ and $q_1=0$; when $p=q=n$ and $\delta=\alpha_i$ we get $n_1=n$, $p_1=p-t_1$ and $q_1=n$; when $p=q=n$ and $\delta=\beta_j$ we get $n_1=n$, $p_1=n$ and $q_1=n-t_1$. Therefore, we always have $p_1-q_1+n_1\not=n$. Then by the same arguments in the proof of \cite[Lemma~2.3]{zhangkorhonen2022}, we may consider the roots of $\overline{f}-\overline{\gamma}=0$ and also $\overline{f}-\omega\overline{\gamma}=0$, where $\omega$ is the $n$-th root of~1, and finally obtain that the equation $\overline{f}-\omega\overline{\gamma}=0$ can have at most finitely many roots, i.e., $\omega\gamma$ is a Picard exceptional rational function of $f$. This implies that $n=2$. Moreover, $N_c=2$ and the two roots of $P(z,f)$ or $Q(z,f)$ are $\pm \gamma$ since for otherwise by Lemma~\ref{basiclemma  dis1} $f$ would have $3$ Picard exceptional rational functions or $2$ Picard exceptional rational functions with one more completely ramified rational function, a contradiction to the inequality \eqref{multiplicityinequality}. But then it follows from \eqref{yanagiharaeq0} that either $0$ or $\infty$ is a Picard exceptional rational function of $f$, a contradiction. Therefore, $\gamma$ cannot be a Picard exceptional rational function of $f$.

Next, we suppose that $\gamma\not\equiv0$ is a completely ramified rational function of $f$ with multiplicity~$m$. We also do the transformations in \eqref{keylemma equa001  diseq4} and get the equation in \eqref{keylemma equa003  diseq6}. Then by the same arguments as in the proof of \cite[Lemma~2.3]{zhangkorhonen2022}, we obtain that the equation $\overline{f}-\omega\overline{\gamma}=0$ can have at most finitely many roots with multiplicities less than~$m$, i.e., $\omega\gamma$ is also a completely ramified rational function of $f$ with multiplicity~$m$.

Finally, we suppose that $0$ is a root of $Q(z,f)$ of order less than $n$. Then there is a $\beta_j$ such that $\omega\beta_j$ is a completely ramified rational function of $f$ with multiplicity~$m\geq 2$, where $\omega$ is the $n$-th root of~1. If $0$ is a Picard exceptional rational function of $f$, then it follows from \eqref{yanagiharaeq0} that the roots of $P(z,f)$ are all Picard exceptional rational functions of $f$, which together with previous discussions shows that $f$ has at least $3$ Picard exceptional rational functions, which is impossible. Thus $0$ cannot be a Picard exceptional rational function of $f$. The proof is complete.

\end{proof}

Corresponding to \cite[Lemma~2.4]{zhangkorhonen2022} in the case $\deg(R(z,f))\not=n$ of equation \eqref{yanagiharaeq0}, we have the following

\begin{lemma}\label{keylemma  dis4}
Let $f$ be a transcendental meromorphic solution of equation \eqref{yanagiharaeq0}. Then $n=2$ or $n=3$.
Moreover, $\alpha_i$ in \eqref{P} with $k_i<n$ and $\beta_j$ in \eqref{Q} with $l_j<n$ are all simple.
\end{lemma}

\renewcommand{\proofname}{Proof.}
\begin{proof}
We consider the cases $q=0$ and $p=q=n$, respectively. If $n\geq 4$, then by Lemma~\ref{keylemma  dis3} a nonzero rational function $\gamma$ cannot be Picard exceptional rational function of $f$. When $q=0$, if $n\geq 4$, then at least one of $\alpha_i$ in \eqref{P} has order $k_i$ not dividing $n$, which with Lemma~\ref{basiclemma  dis1} shows that $\alpha_i$ is a completely ramified rational function of $f$ with multiplicity~$n$. However, since $\alpha_i\not\equiv0$, by Lemma~\ref{keylemma  dis3} $f$ would have~$4$ completely ramified rational functions with multiplicity~$n$, a contradiction to the inequality~\eqref{multiplicityinequality}. Therefore, when $q=0$ we have $n=2$ or $n=3$. When $p=q=n$, we suppose that $n\geq 4$. Recall that $N_c\leq 4$. If some $\alpha_i$ in \eqref{P} has order $k_i$ such that $(n,k_i)<n/2$, then we get a similar contradiction as in the case $q=0$. Since $n\geq 4$, this implies that either $P(z,f)$ has only one root or that $P(z,f)$ has two distinct $\alpha_i$ with two orders $k_i$ satisfying $k_i=n/2$. In the first case, $Q(z,f)$ has at least two distinct roots and none of $\beta_j$ in \eqref{Q} is zero for otherwise by Lemma~\ref{keylemma  dis3} it follows that $f$ has~$5$ completely ramified rational functions, a contradiction to Theorem~\ref{completelyrm}; but then we also have a contradiction as in the case $q=0$ since at least one $\beta_j$ is a completely ramified rational function with multiplicity~$n$. In the latter case, $Q(z,f)$ must have two distinct roots and none of $\beta_j$ in \eqref{Q} is zero for otherwise $f$ would have $5$ completely ramified rational function of $f$, a contradiction to Theorem~\ref{completelyrm}; but then we also have a contradiction as in the case $q=0$ since at least one $\beta_j$ has order $l_j$ not dividing $n$ and thus is a completely ramified rational function with multiplicity~$n$. Therefore, when $p=q=n$, we also have $n=2$ or $n=3$.

Clearly, when $n=2$, $\alpha_i$ in \eqref{P} with $k_i<n$ and $\beta_j$ in \eqref{Q} with $l_j<n$ are all simple. We claim that $\alpha_i$ in \eqref{P} with $k_i<n$ and $\beta_j$ in \eqref{Q} with $l_j<n$ are also simple when $n=3$. In fact, when $n=3$, since $f$ has~$3$ non-zero completely ramified rational functions with multiplicities~$3$ it follows by Yamanoi's second main theorem that $0$ and $\infty$ are both not completely ramified rational functions of $f$. If one $\alpha_i$ or $\beta_j$ in \eqref{P} and \eqref{Q} is not simple, then by a simple analysis as in the proof of \cite[Lemma~2.4]{zhangkorhonen2022} we conclude that there are at least $T(r,f)+o(T(r,f))$ many points $z_0$ such that $f(z_0+1)=0$ or $f(z_0+1)=\infty$ with multiplicity~$m_0\geq 2$ and then by computing $\overline{N}(r,1/\overline{f})$ or $\overline{N}(r,\overline{f})$ as in the proof of \cite[Lemma~2.4]{zhangkorhonen2022} we will get a contradiction. We omit those details.

\end{proof}

Let $\gamma\not\equiv0$ be a completely ramified rational function of $f$ with multiplicity~$m\geq 2$. We further consider the roots of the equation $\overline{f}^n-\overline{\gamma}^n=0$. In particular, we may choose $\gamma=\alpha_i$ or $\gamma=\beta_j$. By Lemma~\ref{keylemma  dis3}, $\omega\gamma$ is a completely ramified rational function of $f$ with multiplicity~$m$, where $\omega$ is an $n$-th root of~1. By \eqref{yanagiharaeq0}, when $q=0$, we have
\begin{equation}\label{keylemma equa0  diseq18}
\overline{f}^n-\overline{\gamma}^n=P(z,f)-\overline{\gamma}^n=a_p(f-\gamma_1)^{t_1}\cdots (f-\gamma_{\tau})^{t_{\tau}},
\end{equation}
or, when $q\geq 1$, we have
\begin{equation}\label{keylemma equa1  diseq19}
\overline{f}^n-\overline{\gamma}^n=\frac{P(z,f)-\overline{\gamma}^nQ(z,f)}{Q(z,f)}=\frac{a_{p_{\tau}}(f-\gamma_1)^{t_1}\cdots (f-\gamma_{\tau})^{t_{\tau}}}{Q(z,f)},
\end{equation}
where $\gamma_1$, $\cdots$, $\gamma_{\tau}$ are in general algebraic functions distinct from each other and $t_1,\cdots,t_{\tau}\in \mathbb{N}$ denote the orders of the roots $\gamma_1,\ldots, \gamma_{\tau}$, respectively, and $t_1+\cdots+t_{\tau}=p_{\tau}\in \mathbb{N}$. In \eqref{keylemma equa0  diseq18} we have $p_{\tau}=n$ and in \eqref{keylemma equa1  diseq19} we have either $p_{\tau}<n$ when $p=q=n$ and $a_p= \overline{\gamma}^n$ or $p_{\tau}=n$ otherwise. We apply the analysis in the proof of \cite[Lemma~2.3]{zhangkorhonen2022} to equations \eqref{keylemma equa0  diseq18} and \eqref{keylemma equa1  diseq19}, respectively, and get the following

\begin{lemma}\label{keylemma1  dis4  diseq20}
Let $f$ be a transcendental meromorphic solution of equation \eqref{yanagiharaeq0}. Suppose that $\gamma\not\equiv0$ is a completely ramified rational function of $f$ with multiplicity~$m\geq 2$. If some $\gamma_i$ in \eqref{keylemma equa0  diseq18} or \eqref{keylemma equa1  diseq19} has order $t_i<m$, then $\gamma_i$ is completely ramified rational function of $f$. In particular, in \eqref{keylemma equa1  diseq19} if $0<q-p_{\tau}<m$, then $\infty$ is a completely ramified rational function of $f$. Further, suppose that $\zeta_i,\ldots,\zeta_t$ are completely ramified rational functions of $f$ such that $\sum_{i=1}^t\Theta(\zeta_i,f)=2$. Then, for each $\gamma_i$ in \eqref{keylemma equa0  diseq18} or \eqref{keylemma equa1  diseq19}, if $\gamma_i$ is not a completely ramified rational function of $f$, then $t_i=m$; if $\gamma_i$ is a completely ramified rational function of $f$ with multiplicity~$m_i\geq 2$, then $t_im_i=m$. In particular, for \eqref{keylemma equa1  diseq19}, when $1\leq p_{\tau}<q$, if $\infty$ is not a completely ramified rational function of $f$, then $q-p_{\tau}=m$; if $\infty$ is a completely ramified rational function of $f$ with multiplicity $m_{\infty}\geq 2$, then $(q-p_{\tau})m_{\infty}=m$.
\end{lemma}

With the above five lemmas, we are ready to derive the eleven equations \eqref{yanagiharaeq11 co}--\eqref{intro_eq_list2_6} from \eqref{yanagiharaeq0}. We make two remarks. First, when $p=q=n=3$, if $P(z,f)$ has only one root $\alpha$ and $Q(z,f)$ has three distinct roots $\beta_1$, $\beta_2$ and $\beta_3$, then by Lemma~\ref{keylemma  dis3} and the inequality \eqref{multiplicityinequality} we see that none of $\beta_j$ is zero, since for otherwise $f$ would have at least four completely ramified rational functions with multiplicity~3, which is impossible. Then the transformation $f\to1/f$ leads \eqref{yanagiharaeq0} to the case that $p=n=3$, $q=0$ when $\alpha\equiv0$ or to the case that $p=q=n=3$ and $Q(z,f)$ has only one root when $\alpha\not\equiv0$. Second, when $p=q=n=2$, if $P(z,f)$ has only one root $\alpha$ and $Q(z,f)$ has two distinct nonzero roots $\beta_1$ and $\beta_2$, then the transformation $f\to1/f$ leads \eqref{yanagiharaeq0} to the case that $p=n=2$, $q=0$ when $\alpha\equiv0$ or to the case that $p=q=n=2$ and $Q(z,f)$ has only one root when $\alpha\not\equiv0$. On the other hand, if one of the two roots $\beta_1$ and $\beta_2$ is zero, then the transformation $f\to 1/f$ leads \eqref{yanagiharaeq0} to the case that $p=1$, $q=2$ and $Q(z,f)$ has only one root. Therefore, by Lemmas~\ref{basiclemma  dis1}, \ref{basiclemma0  dis2} and~\ref{keylemma  dis4}, we see that we only need to consider the following six cases of \eqref{yanagiharaeq0}:
\begin{itemize}
  \item [(1)] $p=n=3$, $q=0$ and $P(z,f)$ has three distinct non-zero roots $\alpha_1$, $\alpha_2$ and $\alpha_3$;
  \item [(2)] $p=q=n=3$, $P(z,f)$ has three distinct non-zero roots $\alpha_1$, $\alpha_2$ and $\alpha_3$ and $Q(z,f)$ has only one root $\beta$;
  \item [(3)] $p=n=2$, $q=0$ and $P(z,f)$ has two distinct non-zero roots $\alpha_1$ and $\alpha_2$;
  \item [(4)] $p=q=n=2$, $P(z,f)$ has two distinct non-zero roots $\alpha_1$ and $\alpha_2$ and $Q(z,f)$ has only one root $\beta$;
  \item [(5)] $p=1$, $q=n=2$, $P(z,f)$ has only one non-zero root $\alpha$ and $Q(z,f)$ has only one root $\beta$;
  \item [(6)] $p=q=n=2$, $P(z,f)$ has two distinct non-zero roots $\alpha_1$ and $\alpha_2$ and $Q(z,f)$ has two distinct roots $\beta_1$ and $\beta_2$.
\end{itemize}
For each of the above six cases, we shall use Lemma~\ref{keylemma1  dis4  diseq20} together with Theorem~\ref{completelyrm} or the inequality~\eqref{multiplicityinequality} to consider \eqref{keylemma equa0  diseq18} and \eqref{keylemma equa1  diseq19} as in \cite{zhangkorhonen2022} to determine the coefficients of $R(z,f)$ of \eqref{yanagiharaeq0}, which then yield the~$11$ equations \eqref{yanagiharaeq11 co}--\eqref{intro_eq_list2_6} after doing a bilinear transformation $f\to\alpha f$ with a suitable algebraic function $\alpha$. Below we apply this strategy to each of the above six cases respectively.

Consider first case (1). By Lemma~\ref{keylemma  dis3}, for each $\alpha_i$, $\omega\alpha_i$ is a completely ramified rational function of $f$ with multiplicity $3$ for any $\omega$ such that $\omega^3=1$. By the inequality \eqref{multiplicityinequality} we may suppose that $\alpha_2=\omega \alpha_1$ and $\alpha_3=\omega^2\alpha_1$ for an $\omega$ such that $\omega^2+\omega+1=0$. Thus, by doing a linear transformation $f\to\alpha_1 f$, we may rewrite equation \eqref{first_order_de_n} as
\begin{equation}\label{yanagiharaeq0    bn1}
\overline{f}^3=c(1-f^3),
\end{equation}
where $c=-a_p\alpha_1^3/\overline{\alpha}_1^3$ is a rational function. By \eqref{yanagiharaeq0    bn1}, we consider
\begin{equation}\label{yanagiharaeq0    bn1 fu}
\overline{f}^3-1=c(1-f^3)-1=-c\left(f^3-\frac{c-1}{c}\right).
\end{equation}
Note that $f$ now has three distinct completely ramified rational functions $1,\omega,\omega^2$ with multiplicity~3 and has no other completely ramified rational functions. By Lemma~\ref{keylemma1  dis4  diseq20}, we must have $c-1=0$. This gives the equation \eqref{intro_eq_list2_5}.

Consider next case (2). By the same arguments as in case (1), we may suppose that $\alpha_2=\omega\alpha_1$ and $\alpha_3=\omega^2\alpha_1$ for an $\omega$ such that $\omega^2+\omega+1=0$. Thus, by doing a linear transformation $f\to\alpha_1 f$, we may rewrite equation \eqref{first_order_de_n} as
\begin{equation}\label{yanagiharaeq0    bn3}
\overline{f}^3=\frac{c(f^{3}-1)}{(f-\delta)^3},
\end{equation}
where $c=a_p/\overline{\alpha}_1^3$ and $\delta=\beta/\alpha_1$ are in general algebraic functions. By \eqref{yanagiharaeq0    bn3}, we consider
\begin{equation}\label{yanagiharaeq0    bn3 fu}
\overline{f}^3-1=\frac{c(f^{3}-1)}{(f-\delta)^3}-1=\frac{c(f^{3}-1)-(f-\delta)^3}{(f-\delta)^3}.
\end{equation}
Note that $f$ now has three distinct completely ramified rational functions $1,\omega,\omega^2$ with multiplicity~3 and has no other completely ramified rational functions. If $c\not=1$, then by Lemma~\ref{keylemma1  dis4  diseq20} we must have $c(f^{3}-1)-(f-\delta)^3=(c-1)(f-\gamma)^3$ for some algebraic function $\gamma$ distinct from $\delta$. However, a simple comparison on the terms of degrees $1$ and $2$ on both sides of this equation yields $\delta=\gamma$, a contradiction. If $c=1$, since the terms of degree $3$ cancel out, then by Lemma~\ref{keylemma1  dis4  diseq20} we must have $\delta=0$ so that $c(f^{3}-1)-(f-\delta)^3$ reduces to be an algebraic function. This gives the equation \eqref{yanagiharaeq15 co}.

Consider next case (3). We claim that $\alpha_1+\alpha_2=0$. Otherwise, by Lemma~\ref{keylemma  dis3}, $f$ has four completely ramified rational functions with multiplicities~2, namely $\pm \alpha_1$ and $\pm\alpha_2$. Now we consider
\begin{equation}\label{yanagiharaeq0    bn3 fu4}
\overline{f}^2-\overline{\alpha}_1^2= a_p(f-\alpha_1)(f-\alpha_2)-\overline{\alpha}_1^2.
\end{equation}
By Lemma~\ref{keylemma1  dis4  diseq20}, if some root of the polynomial $a_p(f-\alpha_1)(f-\alpha_2)-\overline{\alpha}_1^2$ is not equal to $-\alpha_1$ or $-\alpha_2$, then this root has order two. This implies that $-\alpha_1$ and $-\alpha_2$ are either both simple roots of the polynomial $a_p(f-\alpha_1)(f-\alpha_2)-\overline{\alpha}_1^2$, or neither of them are. Note that $a_p(f-\alpha_1)(f-\alpha_2)-\overline{\alpha}_i^2$ cannot be a square of some polynomial in $f$ for both $i=1,2$. By considering $\overline{f}^2-\overline{\alpha}_2^2$ again, then, in the first case we conclude by Lemma~\ref{keylemma1  dis4  diseq20} that the polynomial $a_p(f-\alpha_1)(f-\alpha_2)-\overline{\alpha}_2^2$ is a square of some polynomial in $f$ and in the latter case we conclude by Lemma~\ref{keylemma1  dis4  diseq20} that $-\alpha_1$ and $-\alpha_2$ are both simple roots of the polynomial $a_p(f-\alpha_1)(f-\alpha_2)-\overline{\alpha}_2^2$. We only need to consider the first case. Now, by doing a linear transformation $f\to\alpha_1 f$, we have the system of two equations:
\begin{equation}\label{yanagiharaeq0    bn3 fu5}
\begin{split}
c(f-1)(f-\kappa)-1&= c(f+1)(f+\kappa),\\
c(f-1)(f-\kappa)-\overline{\kappa}^2&= c(f-\delta)^2,
\end{split}
\end{equation}
where $c=a_p\alpha_1^2/\overline{\alpha}_1^2$, $\kappa=\alpha_2/\alpha_1$ and $\delta$ are in general algebraic functions. However, by comparing the coefficients on both sides of the two equations in \eqref{yanagiharaeq0    bn3 fu5}, we deduce from the resulting coefficient relations that $1+\kappa=0$, a contradiction. Therefore, $\alpha_1+\alpha_2=0$. Then, by doing a linear transformation $f\to\alpha_1 f$, we may rewrite equation \eqref{first_order_de_n} as
\begin{equation}\label{yanagiharaeq0    bn4}
\overline{f}^2=c(1-f^2),
\end{equation}
where $c=-a_p\alpha_1^2/\overline{\alpha}_1^2$ is a rational function. By \eqref{yanagiharaeq0    bn4}, we consider
\begin{equation}\label{yanagiharaeq0    bn4  fu1}
\overline{f}^2-1=c(1-f^2)-1=-c\left(f^2-\frac{c-1}{c}\right).
\end{equation}
Note that $f$ now has two distinct completely ramified rational functions $\pm1$. If $c-1=0$, then $c=1$ and we get the equation \eqref{yanagiharaeq11 co}. Otherwise, $c-1\not=0$, then by Lemma~\ref{keylemma1  dis4  diseq20} we see that $\pm \sqrt{(c-1)/c}$ are both completely ramified rational functions of $f$. Again, we consider
\begin{equation}\label{yanagiharaeq0    bn4  fu2}
\overline{f}^2-\frac{\overline{c}-1}{\overline{c}}=c(1-f^2)-\frac{\overline{c}-1}{\overline{c}}=-c\left(f^2-\frac{\overline{c}c-\overline{c}+1}{\overline{c}c}\right).
\end{equation}
Since now $f$ has four completely ramified rational functions $\pm1,\pm\sqrt{(c-1/)c}$, then by Lemma~\ref{keylemma1  dis4  diseq20} we must have $\overline{c}c-\overline{c}+1=0$, i.e., $c=-\omega^2$, where $\omega$ is a constant such that $\omega^2+\omega+1=0$. This gives the equation \eqref{intro_eq_list2_1}.

Consider next case (4). In this case, if $\alpha_1+\alpha_2\not=0$, then by Lemma~\ref{keylemma  dis3}, $f$ has four completely ramified rational functions with multiplicities~2, namely $\pm \alpha_1$ and $\pm\alpha_2$. Suppose first that $a_p\not=\overline{\alpha}_1^2,\overline{\alpha}_2^2$. We consider
\begin{equation}\label{yanagiharaeq0    bn5}
\overline{f}^2-\overline{\alpha}_1^2= \frac{a_p(f-\alpha_1)(f-\alpha_2)-\overline{\alpha}_1^2(f-\beta)^2}{(f-\beta)^2}.
\end{equation}
By the same arguments as in case (3), $-\alpha_1$ and $-\alpha_2$ are either both simple roots of the polynomial $a_p(f-\alpha_1)(f-\alpha_2)-\overline{\alpha}_1^2(f-\beta)^2$, or neither of them are. If $a_p(f-\alpha_1)(f-\alpha_2)-\overline{\alpha}_i^2(f-\beta)^2$ is a square of some polynomial in $f$ for both $i=1,2$, then by computing the two discriminants $\Delta_i:=[a_p(\alpha_1+\alpha_2)-2\overline{\alpha}_i^2\beta]^2-4(a_p-\overline{\alpha}_i^2)(a_p\alpha_1\alpha_2-\overline{\alpha}_i^2\beta^2)=0$, $i=1,2$, we deduce that $(\beta-\alpha_1)(\beta-\alpha_2)=0$, which is impossible. By considering $\overline{f}^2-\overline{\alpha}_2^2$ again, then, in the first case we conclude by Lemma~\ref{keylemma1  dis4  diseq20} that the polynomial $a_p(f-\alpha_1)(f-\alpha_2)-\overline{\alpha}_2^2(f-\beta)^2$ is a square of some polynomial in $f$ and in the latter case we conclude by Lemma~\ref{keylemma1  dis4  diseq20} that $-\alpha_1$ and $-\alpha_2$ are both simple roots of the polynomial $a_p(f-\alpha_1)(f-\alpha_2)-\overline{\alpha}_2^2(f-\beta)^2$. We only need to consider the first case. Now, by doing a linear transformation $f\to\alpha_1 f$, we have the system of two equations:
\begin{equation}\label{yanagiharaeq0    bn5  fu1}
\begin{split}
c(f-1)(f-\kappa)-(f-\delta)^2&= (c-1)(f+1)(f+\kappa),\\
c(f-1)(f-\kappa)-\overline{\kappa}^2(f-\delta)^2&= (c-\overline{\kappa}^2)(f-\gamma)^2,
\end{split}
\end{equation}
where $c=a_p/\overline{\alpha}_1^2$, $\kappa=\alpha_2/\alpha_1$, $\delta=\beta/\alpha_1$ and $\gamma$ are in general algebraic functions. By comparing the coefficients on both sides of the two equations in \eqref{yanagiharaeq0    bn5  fu1}, we deduce from the resulting coefficient relations that $\kappa=\delta^2$, $\gamma=-\delta$, $c=\frac{1}{2}\frac{(1+\delta)^2}{1+\delta^2}$ and $\delta\not=0,\pm1,\pm i$ satisfies  $8\overline{\delta}^4(\delta^2+1)\delta=(\delta+1)^4$. Note that $\delta\equiv1$ solves this equation. We see that $\delta$ is a constant and thus $\alpha_1$ and $\alpha_2$ are both rational functions. This gives the equation \eqref{intro_eq_list2_6}. Now, if $a_p=\overline{\alpha}_1^2$ or $a_p=\overline{\alpha}_2^2$, then by similar discussions as above, we have the system of two equations:
\begin{equation}\label{yanagiharaeq0    bn5  fu1f1}
\begin{split}
(f-1)(f-\kappa)-(f-\delta)^2&= c_1,\\
(f-1)(f-\kappa)-\overline{\kappa}^2(f-\delta)^2&= (1-\overline{\kappa}^2)(f-\gamma)^2,
\end{split}
\end{equation}
or
\begin{equation}\label{yanagiharaeq0    bn5  fu1f2}
\begin{split}
\overline{\kappa}^2(f-1)(f-\kappa)-(f-\delta)^2&= (\overline{\kappa}^2-1)(f+1)(f+\kappa),\\
\overline{\kappa}^2(f-1)(f-\kappa)-\overline{\kappa}^2(f-\delta)^2&= c_2,
\end{split}
\end{equation}
where $\delta=\beta/\alpha_1$, $c_1$, $c_2$ and $\gamma$ are in general algebraic functions. However, by comparing the coefficients on both sides, we deduce from the resulting coefficient relations obtained from the system of two equations in \eqref{yanagiharaeq0    bn5  fu1f1} that $\delta=\gamma$ and the resulting coefficient relations obtained from the system of two equations in \eqref{yanagiharaeq0    bn5  fu1f2} that $1+\kappa=0$, both of which are impossible. On the other hand, if $\alpha_1+\alpha_2=0$, then by doing a linear transformation $f\to\alpha_1 f$, we may rewrite equation \eqref{first_order_de_n} as
\begin{equation}\label{yanagiharaeq0    bn5  fu2}
\overline{f}^2=\frac{c(1-f^2)}{(f-\delta)^2},
\end{equation}
where $c=-a_p/\overline{\alpha}_1^2$ and $\delta=\beta/\alpha_1$ are in general algebraic functions. By \eqref{yanagiharaeq0    bn5  fu2}, we consider
\begin{equation}\label{yanagiharaeq0    bn5  fu3}
\overline{f}^2-1=\frac{c(1-f^2)}{(f-\delta)^2}-1=\frac{c(1-f^2)-(f-\delta)^2}{(f-\delta)^2}.
\end{equation}
When $c=-1$, if $\delta\not=0$, then by Lemma~\ref{keylemma1  dis4  diseq20} we see that $2\delta/(\delta^2+1)$ and $\infty$ are both completely ramified rational functions of $f$. However, by Lemma~\ref{keylemma  dis3} $f$ would have $5$ completely ramified rational functions, a contradiction to Theorem~\ref{completelyrm}. Therefore, we must have $\delta=0$ when $c=-1$. When $c\not=-1$, if $c(1-f^2)-(f-\delta)^2$ has only one root, i.e., the discriminant $\Delta:=4\delta^2+4(c+1)(c-\delta^2)=0$, then $c=\delta^2-1$. The above two cases give the equation \eqref{yanagiharaeq12 co}. Otherwise, we have $c\not=-1$ and $c(1-f^2)-(f-\delta)^2$ has two distinct roots, say $\delta_1$ and $\delta_2$, which are both completely ramified rational functions of $f$ by Lemma~\ref{keylemma1  dis4  diseq20}. By Lemma~\ref{keylemma  dis3}, $\pm \delta_1$ and $\pm \delta_2$ are all completely ramified rational functions of $f$. By Theorem~\ref{completelyrm} we must have $\delta_1+\delta_2=0$. It follows that $\delta=0$ and $\delta_1^2=c/(c+1)$. Again, we consider
\begin{equation}\label{yanagiharaeq0    bn5  fu4}
\overline{f}^2-\frac{\overline{c}}{\overline{c}+1}=\frac{c(1-f^2)}{f^2}-\frac{\overline{c}}{\overline{c}+1}=\frac{c-[(\overline{c}c+c+\overline{c})/(\overline{c}+1)]f^2}{f^2}.
\end{equation}
If $\overline{c}c+c+\overline{c}\not=0$, then by Lemma~\ref{keylemma1  dis4  diseq20} we see that $\pm \sqrt{c(\overline{c}+1)/(\overline{c}c+c+\overline{c})}$ are both completely ramified rational functions of $f$, a contradiction to Theorem~\ref{completelyrm}. Therefore, we must have $\overline{c}c+c+\overline{c}=0$ and thus $c=-2$. This gives the equation \eqref{intro_eq_list2_2}.

Consider next case (5). In this case, by doing a linear transformation $f\to -\alpha f$, we may rewrite equation \eqref{first_order_de_n} as
\begin{equation}\label{yanagiharaeq0    bn6}
\overline{f}^2=\frac{c(f+1)}{(f-\delta)^2},
\end{equation}
where $c=-a_p/\alpha\overline{\alpha}^2$ and $\delta=-\beta/\alpha$ are rational functions. Then from previous discussions and Lemmas~\ref{basiclemma  dis1} and \ref{keylemma  dis3}, we see that $\infty$, $\pm1$ and $\pm \delta$ are all completely ramified rational functions of $f$. By Theorem~\ref{completelyrm}, we must have $\delta=1$. By \eqref{yanagiharaeq0    bn6}, we consider
\begin{equation}\label{yanagiharaeq0    bn6  fu1}
\overline{f}^2-1=\frac{c(f+1)}{(f-1)^2}-1=\frac{c(f+1)-(f-1)^2}{(f-1)^2}.
\end{equation}
If $c(f+1)-(f-1)^2$ has two distinct roots, then by Lemma~\ref{keylemma1  dis4  diseq20} these two roots are both completely ramified rational functions, a contradiction to Theorem~\ref{completelyrm}. Thus $c(f+1)-(f-1)^2$ can have only one root, which implies that the discriminant $\Delta:=(c+2)^2+4(c-1)=0$, i.e., $c=-8$. This gives the equation \eqref{yanagiharaeq13 co}.

Consider finally case (6). By Lemma~\ref{basiclemma  dis1}, $\pm\alpha_1$, $\pm\alpha_2$, $\pm\beta_1$ and $\pm\beta_2$ are all completely ramified rational functions of $f$. By Theorem~\ref{completelyrm}, we must have either $\alpha_1+\alpha_2=0$ and $\beta_1+\beta_2=0$ or $\alpha_1+\beta_1=0$ and $\alpha_2+\beta_2=0$. When $\alpha_1+\alpha_2=0$ and $\beta_1+\beta_2=0$, by doing a linear transformation $f\to\beta_1 f$, we may rewrite equation \eqref{first_order_de_n} as
\begin{equation}\label{yanagiharaeq0    bn7}
\overline{f}^2=\frac{c(f^2-\kappa^2)}{f^2-1},
\end{equation}
where $c=a_p/\overline{\beta}_1^2$ and $\kappa=\alpha_1/\beta_1$ are in general algebraic functions. We see that $\kappa^2\not=0,1$. By \eqref{yanagiharaeq0    bn6}, we consider
\begin{equation}\label{yanagiharaeq0    bn7  fu1}
\overline{f}^2-1=\frac{c(f^2-\kappa^2)}{f^2-1}-1=\frac{(c-1)f^2-(c\kappa^2-1)}{f^2-1}.
\end{equation}
If $c\not=1$ and $c\not=1/\kappa^2$, then by Lemma~\ref{keylemma1  dis4  diseq20}, $\pm \sqrt{(c\kappa^2-1)/(c-1)}$ are both completely ramified rational functions of $f$ and thus $f$ would have 6 completely ramified rational functions, a contradiction to Theorem~\ref{completelyrm}. Therefore, $c=1$ or $c=1/\kappa^2$. If $c=1$, then we get the equation \eqref{yanagiharaeq14 co}; if $c=1/\kappa^2$, then we consider the equation $\overline{f}^2-\overline{\kappa}^2$ and by the same arguments as above to obtain that $1/\kappa^2=\overline{\kappa}^2$, i.e., $\kappa^2=-1$ and $c=-1$ and thus we get the equation \eqref{intro_eq_list2_3}. On the other hand, when $\alpha_1+\beta_1=0$ and $\alpha_2+\beta_2=0$, by doing a linear transformation $f\to \sqrt{\alpha_1\alpha_2} f$, we may rewrite equation \eqref{first_order_de_n} as
\begin{equation}\label{yanagiharaeq0    bn8}
\overline{f}^2=c\frac{(f-\delta)(f-\delta^{-1})}{(f+\delta)(f+\delta^{-1})},
\end{equation}
where $c=a_p/\overline{\alpha}_1\overline{\alpha}_2$ and $\delta=\sqrt{\alpha_1/\alpha_2}$ are in general algebraic functions. By the same arguments as for the equation \eqref{yanagiharaeq0    bn5}, we may consider $\overline{f}^2-\overline{\delta}^2$ and conclude by Lemma~\ref{keylemma1  dis4  diseq20} that the polynomial $c(f-\delta)(f-\delta^{-1})-\overline{\delta}^2(f+\delta)(f+\delta^{-1})$ is a square of some polynomial in $f$ when $c\not=\overline{\delta}^2$ or reduces to be an algebraic function $c_1$ when $c=\overline{\delta}^2$. Since $\delta^2\not=0,\pm1, \pm i$, a straightforward comparison shows that the latter case is impossible. Similarly, we may consider $\overline{f}^2-\overline{\delta}^{-2}$ and conclude that the polynomial $c(f-\delta)(f-\delta^{-1})-\overline{\delta}^{-2}(f+\delta)(f+\delta^{-1})$ is a square of some polynomial in $f$. Thus we have the system of two equations:
\begin{equation}\label{yanagiharaeq0    bn8   fu1}
\begin{split}
c(f-\delta)(f-\delta^{-1})-\overline{\delta}^2(f+\delta)(f+\delta^{-1})&= (c-\overline{\delta}^2)(f-\gamma_1)^2,\\
c(f-\delta)(f-\delta^{-1})-\overline{\delta}^{-2}(f+\delta)(f+\delta^{-1})&= (c-\overline{\delta}^{-2})(f-\gamma_2)^2,
\end{split}
\end{equation}
where $\gamma_1$ and $\gamma_2$ are in general algebraic functions. By comparing the coefficients on both sides of the two equations in \eqref{yanagiharaeq0    bn8   fu1}, we deduce from the resulting coefficient relations that $\gamma_1^2=\gamma_2^2=c^2=1$ and $d=\delta+\delta^{-1}$ satisfies  $\overline{d}^2(d^2-4)=2(1-c)d^2-8(1+c)$. We see that $d$ is a constant. This gives the equation \eqref{intro_eq_list2_4} and also completes the classification for equation \eqref{yanagiharaeq0}.

\subsection{Growth of meromorphic solutions of equation \eqref{intro_eq_list2_6}}
We show that all transcendental meromorphic solutions $f$ of equation \eqref{intro_eq_list2_6} have hyper-order $\geq 1$. Note that $f$ is twofold ramified over each of $\pm1,\pm \delta^2$. Then there exists an entire function $\varphi(z)$ such that $f$ is written as $f(z)=\text{sn}(\varphi(z))$, where $\text{sn}(\varphi)=\text{sn}(\varphi,1/\delta^2)$ is the Jacobi elliptic function with the modulus $1/\delta^2$ and satisfies the second order differential equation $\text{sn}'(\varphi)^2=(1-\text{sn}(\varphi)^2)(1-\text{sn}(\varphi)^2/\delta^4)$. Moreover, by the second main theorem we have $T(r,f)=N(r,1/(f-1))+O(\log r)$. Let $z_0$ be a point such that $f(z_0)=\text{sn}(\varphi(z_0))=1$. It follows from \eqref{intro_eq_list2_6} that $f(z_0+1)=\text{sn}(\varphi(z_0+1))=0$. Computing the Maclaurin series for $\text{sn}(\varphi)$ and $\text{sn}(\overline{\varphi})$ around the point $z_0$, respectively, we get
    \begin{equation}\label{Poiu1}
    \begin{split}
    \text{sn}(\varphi(z_0))=1-\frac{\delta^4-1}{\delta^4}(\varphi(z)-\varphi(z_0))+\cdots=1-\frac{\delta^4-1}{\delta^4}\varphi'(z_0)(z-z_0)+\cdots,
    \end{split}
    \end{equation}
and
    \begin{equation}\label{Poiu2}
    \begin{split}
   \text{sn}(\varphi(z_0+1))=\varphi(z+1)-\varphi(z_0+1)+\cdots=\varphi'(z_0+1)(z-z_0)+\cdots.
    \end{split}
    \end{equation}
By substituting the above two expressions into \eqref{intro_eq_list2_6} and then comparing the second-degree terms on both sides of the resulting equation, we find
    \begin{equation}\label{Poiu3}
    \varphi'(z_0+1)^2 = \frac{1}{2}\frac{(1+\delta)^4}{\delta^4}\varphi'(z_0)^2.
    \end{equation}
A simple computation together with equation \eqref{Poiu} shows that $(1+\delta)^4\not=2\delta^4$. Define $G(z):=\varphi'(z+1)^2- \frac{1}{2}\frac{(1+\delta)^4}{\delta^4}\varphi'(z)^2$. From the discussions in \cite{Korhonenzhang2020} we know that $T(r,\overline{f})=T(r,f)+O(\log r)$. Since $\text{sn}(z)$ has positive order of growth, then by \cite[p.~50]{Hayman1964Meromorphic} we have $T(r,\varphi)=o(T(r,f))$ and $T(r,\overline{\varphi})=o(T(r,\overline{f}))$, where $r\to\infty$. If $G\not\equiv0$, then $T(r,G)\leq o(T(r,f))$, $r\to\infty$, which is impossible since $G$ has  $T(r,f)+O(\log r)$ many zeros. Thus $G(z)\equiv0$. Now, $\varphi'(z+1) =\pm \frac{1}{\sqrt{2}}\frac{(1+\delta)^2}{\sigma^2}\varphi'(z)$ and by integration we see that $\varphi$ is an entire function such that $T(r,\varphi)\geq Kr$ for some $K>0$ and all $r\geq r_0$ with some $r_0\geq 0$.
Since $\text{sn}(z)$ has positive exponent of convergence of zeros and $f(z)=\text{sn}(\varphi(z))$, then the fact that $f$ is of hyper-order at least one is a consequence of Lemma~\ref{hyper-order lemma} below, which is a slightly modified version of \cite[Lemma~5.20]{Laine1993}.

\begin{lemma}\label{hyper-order lemma}
Let $g$ be a meromorphic function such that the exponent of convergence $\lambda=\lambda(g)>0$, and let $\varphi=\varphi(z)$ be an entire function such that  $T(r,\varphi)\geq cr$ for some $c>0$ and all $r\geq r_0$ with some $r_0\geq 0$. Then the hyper order of $g\circ\varphi$ is at least one.
\end{lemma}
\begin{proof}
We consider the zeros of $g\circ\varphi$. Since $\varphi$ is an entire function, then $\varphi$ has at most one finite Picard's exceptional value. Thus we may choose a constant $r_1\geq r_0$ such that $\varphi$ takes in $|z|<t$ every value $w$ in the annulus $r_1<|w|<M(t,\varphi)$, provided that $t$ is large enough. Let $g$ have $\mu(t)$ zeros in this annulus, counted according to their multiplicity. Then by the definition of $\lambda$, we have
    \begin{equation}\label{exponent of zeros}
    \limsup_{r\to\infty}\frac{\log n(r)}{\log r}=\limsup_{t\to\infty}\frac{\log \mu(t)}{\log M(t,\varphi)}=\lambda>0.
    \end{equation}
Hence, for some $\tau>0$, there exists a sequence $(t_n)$ tending to $+\infty$ such that
    \begin{equation}\label{exponent of zeros1}
    \mu(t_n)>\left(M(t_n,\varphi)\right)^{\tau}\geq \left(e^{ct_n}\right)^{\tau},
    \end{equation}
where $c$ is a positive constant. The second inequality above follows by our assumption since $\log M(t,\varphi)\geq T(t,\varphi)$ for all large $t$. Now, $g\circ\varphi$ has at least $\mu(t)$ zeros in $|z|<t$. Making using of \eqref{exponent of zeros1}, we have
    \begin{equation}\label{exponent of zeros2}
    \limsup_{r\to\infty}\frac{\log\log n(r,1/g\circ\varphi)}{\log r}\geq \limsup_{t_n\to\infty}\frac{\log\log \mu(t_n,1/g\circ\varphi)}{\log t_n}\geq 1.
    \end{equation}
By the fact that $T(r,1/g\circ\varphi)\geq N(r,1/g\circ\varphi)$, we conclude that the hyper order of $g\circ\varphi$ is at least one. Thus our assertion follows.
\end{proof}

We also note that the fact that all meromorphic solutions of each equation in the list \eqref{intro_eq_list2_1}--\eqref{intro_eq_list2_5} can be proved using the above method since all solutions of them are elliptic functions composed with entire functions.

\section{Relations between equations \eqref{first_order_di_n} and \eqref{first_order_de_n} when $n=2$}\label{Relations between Malmquist type differential and difference equations}

In this section, we use the bilinear method and the continuum limit method to study the relations between equations \eqref{first_order_di_n} and \eqref{first_order_de_n} for the case $n=2$. For the description of the bilinear method, see \cite{HirotaRI} or \cite{Hietarinta1996}. Here we provide a brief overview of the continuum limit method: Let $k$ be a positive integer, and $\varepsilon$ be a complex number. We set a pair of relations:
\begin{equation}\label{conlim1}
\begin{split}
\mu(z,t,\varepsilon)=0, \quad \nu(f(z),w(t,\varepsilon),\varepsilon)=0.
\end{split}
\end{equation}
According to this, we transform a difference equation
\begin{equation}\label{conlim2}
\Omega_0(z,f(z+1),\cdots,f(z+k))=0
\end{equation}
to a certain difference equation
\begin{equation}\label{conlim3}
\Omega_1(t,w(t,\varepsilon),\cdots,w(t+k\varepsilon,\varepsilon))=0.
\end{equation}
Letting $\varepsilon\to0$, with some conditions on coefficients of $\Omega_1$, we derive a differential equation:
\begin{equation}\label{conlim4}
\Omega_1(t,w'(t,0),\cdots,w^{(k)}(t,0))=0.
\end{equation}
It is clear that the first order linear difference equation has a continuum limit to the first order linear differential equation in the autonomous case. In the two subsections below, we describe the relations between the difference equations \eqref{yanagiharaeq11 co}--\eqref{yanagiharaeq14 co} and the differential equations \eqref{riccati}, \eqref{MYS1} or \eqref{MYS2}. We also show how to take the continuum limit for the five equations \eqref{intro_eq_list2_1}--\eqref{intro_eq_list2_4} and \eqref{intro_eq_list2_6}.

\subsection{Relations between \eqref{yanagiharaeq12 co} and \eqref{yanagiharaeq13 co} and the Riccati equation \eqref{riccati}}
In \cite{Ishizaki2017}, Ishizaki discussed the relation between a differential Riccati equation and a difference Riccati equation. We first recall Ishizaki's results below. For the differential Riccati equation \eqref{riccati}, we assume that $a_2\not\equiv0$ from now on. It is elementary to show that a suitable linear transformation on $f$ leads equation \eqref{riccati} to
\begin{equation}\label{riccati  conlim1}
f'=f^2+A(z),
\end{equation}
where $A(z)$ is a rational function formulated in terms of $a_j$ and their derivatives; see \cite[chapter~9]{Laine1993}. Ishizaki used the bilinear method to derive a difference Riccati equation from \eqref{riccati  conlim1} in the following way: Setting $f(z)=u(z)/v(z)$, then equation \eqref{riccati  conlim1} becomes
\begin{equation}\label{riccati  conlim2fu1}
\begin{split}
u'(z)v(z)-u(z)v'(z) = u(z)^2+A(z)v(z)^2,
\end{split}
\end{equation}
which is gauge invariant. In other words, for any $h(z)$, $\tilde{u}(z)=u(z)h(z)$ and $\tilde{v}(z)=v(z)h(z)$ also satisfy the differential equation \eqref{riccati  conlim2fu1} in place of $u(z)$ and $v(z)$, respectively. Corresponding to equation \eqref{riccati  conlim2fu1}, we choose a difference equation
\begin{equation}\label{riccati  conlim2fu2}
\begin{split}
u(z+1)v(z)-u(z)v(z+1)= u(z)u(z+1)+A(z)v(z)v(z+1),
\end{split}
\end{equation}
having the property of being gauge invariant. Setting $f(z)=u(z)/v(z)$ in the difference equation above, then we obtain
\begin{equation}\label{riccati  conlim2fu3}
\begin{split}
f(z+1)-f(z) = f(z+1)f(z)+A(z),
\end{split}
\end{equation}
i.e.,
\begin{equation}\label{riccati  conlim2fu4}
\begin{split}
f(z+1)=\frac{f(z)+A(z)}{1-f(z)}.
\end{split}
\end{equation}
On the other hand, for the difference Riccati equation \eqref{driccati0}, Ishizaki showed that if $b_1(z)\not=-b_3(z+1)$, by doing the transformation $f(z)\to [(-b_3-\underline{b}_1)f+(-b_3+\underline{b}_1)]/2$ we obtain the difference equation \eqref{riccati  conlim2fu4} with
\begin{equation}\label{riccati  conlim6}
A(z)=\frac{-4b_2-b_1\underline{b}_1+3b_1b_3-\underline{b}_1\overline{b}_3-b_3\overline{b}_3}{(b_3+\underline{b}_1)(\overline{b}_3+b_1)}.
\end{equation}
Set
\begin{equation}\label{riccati  conlim8}
\begin{split}
t=\varepsilon z,\quad f=\varepsilon w(t,\varepsilon),
\end{split}
\end{equation}
with the condition $A(z)=\varepsilon^2 \tilde{A}(t,\varepsilon)$ and $\lim_{\varepsilon\to0}\tilde{A}(t,\varepsilon)=\tilde{A}(t,0)$. Since $f(z+1)=\varepsilon w(\varepsilon (z+1),\varepsilon)=\varepsilon w(t+\varepsilon,\varepsilon)$, we have
\begin{equation}\label{riccati  conlim9}
w(t+\varepsilon,\varepsilon)-w(t,\varepsilon)=\varepsilon w(t+\varepsilon,\varepsilon)w(t,\varepsilon)+\varepsilon \tilde{A}(t,\varepsilon).
\end{equation}
By letting $\varepsilon\to 0$, we have
\begin{equation}\label{riccati  conlim10}
w'(t,0)=w(t,0)^2+\tilde{A}(t,0),
\end{equation}
which is equation \eqref{riccati  conlim1}. In particular, if $A$ is a constant, then we replace $A$ with $\varepsilon^2\tilde{A}$ with a constant $\tilde{A}$.

With the introduction above, let's look at the two equations \eqref{yanagiharaeq12 co} and \eqref{yanagiharaeq13 co}, respectively. For equation \eqref{yanagiharaeq12 co}, if we put $f=(\gamma+\gamma^{-1})/2$, then it follows that
    \begin{equation*}
    \overline{\gamma}^2\pm2i\frac{\delta\gamma^2-2\gamma+\delta}{\gamma^2-2\delta\gamma+1}\overline{\gamma}-1=0.
    \end{equation*}
Solving the equation above, we get four different difference Riccati equations:
    \begin{equation*}
    \overline{\gamma}=\left\{-\theta\frac{(\pm i\delta-\sqrt{1-\delta^2})\gamma\pm i}{\gamma-\delta\pm i\sqrt{1-\delta^2}}\right\}^{\theta}, \quad \theta=\pm1.
    \end{equation*}
It is easy to see that these four difference equations do not have any common solutions. Take the following difference Riccati equation as an example:
    \begin{equation}\label{riccati  conlim11fua1fat}
    \overline{\gamma}=\frac{(-i\delta+\sqrt{1-\delta^2})\gamma-i}{\gamma+(-\delta+i\sqrt{1-\delta^2})},
    \end{equation}
consider the case where $\delta$ is a constant. If $2\delta^2\not=1$, then by doing the transformation $\gamma\to \frac{(1+i)(\delta-\sqrt{1-\delta^2})}{2}\gamma+\frac{(1-i)(\delta+\sqrt{1-\delta^2})}{2}$, we obtain from the equation above that
\begin{equation}\label{riccati  conlim11fua1}
\overline{\gamma}-\gamma=\overline{\gamma}\gamma+A,
\end{equation}
where $A$ has the form in \eqref{riccati  conlim6} with $b_1=(-i\delta+\sqrt{1-\delta^2})$, $b_2=-i$ and $b_3=-\delta+i\sqrt{1-\delta^2}$. Therefore, for the solutions of \eqref{yanagiharaeq12 co} such that \eqref{riccati  conlim11fua1fat} hold, we set
\begin{equation}\label{riccati  conlim11fua2}
\begin{split}
t&=\varepsilon z,\\
f&=\frac{1}{2}\left(\gamma+\frac{1}{\gamma}\right),\\
\gamma&=\frac{(1+i)(\delta-\sqrt{1-\delta^2})}{2}\varepsilon w(t,\varepsilon)+\frac{(1-i)(\delta+\sqrt{1-\delta^2})}{2},
\end{split}
\end{equation}
and replace $A$ with $\varepsilon^2 A$. Then we have the equation in \eqref{riccati  conlim9} and, by letting $\varepsilon\to 0$, we finally obtain the equation in \eqref{riccati  conlim10}. Equation \eqref{yanagiharaeq13 co} is dealt with in a similar way. From the results in \cite{Korhonenzhang2020}, if we put $f=\frac{1-u^2}{u^2}$, $\sqrt{2}u=\frac{1}{2}(\gamma+\gamma^{-1})$, then we also get four different difference Riccati equations which do not have any common solutions. Consider the following case:
    \begin{equation*}
    f=\frac{1-u^2}{u^2}=\frac{8\gamma^2-(\gamma^2+1)^2}{(\gamma^2+1)^2},  \quad \overline{\gamma}=-\frac{-(1+\sqrt{2})\gamma+i}{\gamma-i+i\sqrt{2}}.
    \end{equation*}
By making similar substitutions as in \eqref{riccati  conlim11fua2}, we may obtain the equation in \eqref{riccati  conlim9} and, by letting $\varepsilon\to 0$, we obtain the equation in \eqref{riccati  conlim10}.

\subsection{Relations between \eqref{yanagiharaeq11 co}, \eqref{yanagiharaeq14 co}, \eqref{intro_eq_list2_1}--\eqref{intro_eq_list2_4} and \eqref{intro_eq_list2_6} and the differential equations \eqref{MYS1} or \eqref{MYS2}}
The autonomous versions of the seven difference equations \eqref{yanagiharaeq11 co}, \eqref{yanagiharaeq14 co}, \eqref{intro_eq_list2_1}--\eqref{intro_eq_list2_4} and \eqref{intro_eq_list2_6} are closely related to the QRT map~\cite{QuispelRobertsThompson1988,QuispelRobertsThompson1989}, which is defined by the system of two equations:
\begin{eqnarray}
x_{n+1}&=&\frac{f_1(y_n)-x_nf_2(y_n)}{f_2(y_n)-x_nf_3(y_n)},\label{QRTmapI}\\
y_{n+1}&=&\frac{g_1(x_{n+1})-y_ng_2(x_{n+1})}{g_2(x_{n+1})-y_ng_3(x_{n+1})},\label{QRTmapII}
\end{eqnarray}
where
\begin{eqnarray}
       \left(
         \begin{array}{c}
           f_1(x) \\
           f_2(x) \\
           f_3(x) \\
         \end{array}
       \right)
&=&\left(
     \begin{array}{ccc}
       \alpha_0 & \beta_0 & \gamma_0 \\
       \delta_0 & \varepsilon_0 & \zeta_0 \\
       \kappa_0 & \lambda_0 & \mu_0 \\
     \end{array}
     \right)
     \left(
         \begin{array}{c}
           x^2 \\
           x \\
           1 \\
         \end{array}
       \right)
\times
\left(
     \begin{array}{ccc}
       \alpha_1 & \beta_1 & \gamma_1 \\
       \delta_1 & \varepsilon_1 & \zeta_1 \\
       \kappa_1 & \lambda_1 & \mu_1 \\
     \end{array}
     \right)
     \left(
         \begin{array}{c}
           x^2 \\
           x \\
           1 \\
         \end{array}
       \right),   \label{QRTmapIII}\\
       \left(
         \begin{array}{c}
           g_1(x) \\
           g_2(x) \\
           g_3(x) \\
         \end{array}
       \right)
&=&\left(
     \begin{array}{ccc}
       \alpha_0 & \delta_0 & \kappa_0 \\
       \beta_0 & \varepsilon_0 & \lambda_0 \\
       \gamma_0 & \zeta_0 & \mu_0 \\
     \end{array}
     \right)
     \left(
         \begin{array}{c}
           x^2 \\
           x \\
           1 \\
         \end{array}
       \right)
\times
\left(
     \begin{array}{ccc}
       \alpha_1 & \delta_1 & \kappa_1 \\
       \beta_1 & \varepsilon_1 & \lambda_1 \\
       \gamma_1 & \zeta_1 & \mu_1 \\
     \end{array}
     \right)
     \left(
         \begin{array}{c}
           x^2 \\
           x \\
           1 \\
         \end{array}
       \right),\label{QRTmapVI}
\end{eqnarray}
where '$\times$' denotes the cross product of two vectors. In the symmetric case, i.e.,
\begin{eqnarray}
\left(
     \begin{array}{ccc}
       \alpha_i & \beta_i & \gamma_i \\
       \delta_i & \varepsilon_i & \zeta_i \\
       \kappa_i & \lambda_i & \mu_i \\
     \end{array}
     \right)
=   \left(
     \begin{array}{ccc}
       \alpha_i & \delta_i & \kappa_i \\
       \beta_i & \varepsilon_i & \lambda_i \\
       \gamma_i & \zeta_i & \mu_i \\
     \end{array}
     \right), \quad i=0,1,\label{QRTmapV}
\end{eqnarray}
the QRT family reduces into a single equation
\begin{eqnarray}
w_{n+1}&=&\frac{f_1(w_n)-w_{n-1}f_2(w_n)}{f_2(w_n)-w_{n-1}f_3(w_n)},\label{QRTmapIV}
\end{eqnarray}
by taking $x_n=w_{2n}$ and $y_n=w_{2n+1}$ in \eqref{QRTmapI} and \eqref{QRTmapII}. The symmetric QRT family possesses an invariant:
\begin{equation}\label{MYS2 fu5qret1}
\begin{split}
(\alpha_0&+K\alpha_1)x_{n+1}^2x_n^2+(\beta_0+K\beta_1)(x_{n+1}^2x_n+x_{n+1}x_n^2)+(\gamma_0+K\gamma_1)(x_{n+1}^2+x_n^2)\\
&+(\varepsilon_0+K\varepsilon_1)x_{n+1}x_n+(\zeta_0+K\zeta_1)(x_{n+1}+x_n)+(\mu_0+K\mu_1)=0,
\end{split}
\end{equation}
where $K$ is a constant. By doing a M\"obius transformation $x_n\to \frac{\alpha_1x_n+\beta_1}{\alpha_2x_n+\beta_2}$ with suitable constants $\alpha_i$ and $\beta_j$, the symmetric QRT map in \eqref{MYS2 fu5qret1} takes the form:
\begin{equation}\label{MYS2 fu5qret1f}
\begin{split}
\alpha x_{n+1}^2x_n^2+\beta(x_{n+1}^2+x_n^2)+\gamma x_{n+1}x_n+\delta=0,
\end{split}
\end{equation}
where $\alpha$, $\beta$, $\gamma$ and $\delta$ are constants; see \cite{RamaniCarsteaGrammaticosOhta2002}. By reinterpreting discrete equations as difference equations (see \cite{AblowitzHalburdHerbst2000}), we see later that \eqref{yanagiharaeq14 co} reduces to a special case of the symmetric QRT map in the generic case (i.e., $\alpha\delta\not=0$) and equation \eqref{yanagiharaeq11 co} is the symmetric QRT map in the degenerate case (i.e., $\alpha\delta=0$). Moreover, the five equations \eqref{intro_eq_list2_1}--\eqref{intro_eq_list2_4} and \eqref{intro_eq_list2_6} can also be mapped to the symmetric QRT map, as is shown below.

Suppose that $a$ and $b$ in \eqref{MYS1} and \eqref{MYS2} are both constants. For simplicity, we treat equation \eqref{MYS1} as a special case of equation \eqref{MYS2} with $\tau_1=\tau_3$. By doing a M\"obius transformation $f\to \frac{\alpha_1f+\beta_1}{\alpha_2f+\beta_2}$ with suitable constants $\alpha_i$ and $\beta_j$, if necessary, we may assume that $\tau_1+\tau_3=0$ and $\tau_2+\tau_4=0$. Thus we may write \eqref{MYS2} as
\begin{equation}\label{MYS2 fu1}
\begin{split}
(f')^2 = a(f^2-\tau_1^2)(f^2-\tau_2^2),
\end{split}
\end{equation}
where $\tau_1^2\not=\tau_2^2$ and it is possible that $\tau_1^2=0$. Now the bilinear method applies. Following Ishizaki, we set $f(z)=u(z)/v(z)$ and obtain from equation \eqref{MYS2 fu1} that
\begin{equation}\label{MYS2 fu2}
\begin{split}
[u'(z)v(z)-u(z)v'(z)]^2 = a(z)[u(z)^2-\tau_1^2v(z)^2][u(z)^2-\tau_2^2v(z)^2],
\end{split}
\end{equation}
which is gauge invariant. Corresponding to this equation, we choose a difference equation
\begin{equation}\label{MYS2 fu3}
\begin{split}
&[u(z+1)v(z)-u(z)v(z+1)]^2 \\
= &\, a(z)[u(z)u(z+1)-\tau_1^2v(z)v(z+1)][u(z)u(z+1)-\tau_2^2v(z)v(z+1)],
\end{split}
\end{equation}
having the property of being gauge invariant. Setting $f(z)=u(z)/v(z)$ in the difference equation above, we have
\begin{equation}\label{MYS2 fu4}
\begin{split}
[f(z+1)-f(z)]^2 = a[f(z+1)f(z)-\tau_1^2][f(z+1)f(z)-\tau_2^2].
\end{split}
\end{equation}
which is a special case of the symmetric QRT map \eqref{MYS2 fu5qret1} after expansion.

When $\tau_1^2=0$, we do the transformation $f\to 1/f$ and obtain from \eqref{MYS2 fu4} that
\begin{equation}\label{MYS2 fu4ka1}
\begin{split}
[f(z+1)-f(z)]^2 = -a\tau_2^2\left(f(z+1)f(z)-1/\tau_2^2\right).
\end{split}
\end{equation}
We see that equation \eqref{yanagiharaeq11 co} is included in \eqref{MYS2 fu4ka1}. By setting $t=\varepsilon z$, $f(z)=w(t,\varepsilon)$ and giving $\varepsilon^2(-a\tau_2^2)$ in place of $-a\tau_2^2$, then equation \eqref{MYS2 fu4ka1} has a continuum limit to the equation $(f')^2 = -a(\tau_2^2 f^2-1)$; see~\cite{IshizakiKorhonen2018}. The equation $(f')^2 = -a(\tau_2^2 f^2-1)$ can be obtained from equation \eqref{MYS2 fu1} by doing the transformation $f\to 1/f$.

When $\tau_1^2\not=0$, we re-scale $f$ by $f\to (\tau_1\tau_2)^{1/2}f$ and then expand \eqref{MYS2 fu4} to obtain the canonical form of the symmetric QRT map:
\begin{equation}\label{MYS2 fu5}
\begin{split}
\overline{f}^2f^2+A(\overline{f}^2+f^2)+2B\overline{f}f+1=0,
\end{split}
\end{equation}
where $A=-1/(a\tau_1\tau_2)$ and $B=[2-a(\tau_1^2+\tau_2^2)]/(2a\tau_1\tau_2)$. In particular, for equation \eqref{yanagiharaeq14 co}, we may re-scale $f$ by $f\to f/\kappa_1$ with a constant $\kappa_1$ first to obtain the equation
\begin{equation}\label{MYS2 fu5jh7}
\begin{split}
\overline{f}^2f^2-\kappa_1^2(\overline{f}^2+f^2)+\kappa_2^2=0,
\end{split}
\end{equation}
where $\kappa_2^2=\kappa_1^4\kappa^2$. By doing the transformation $f\to \alpha\frac{f-\beta}{f+\beta}$ with constants $\alpha$, $\beta$ satisfying $\alpha^4=\kappa_2^2$, we obtain from \eqref{MYS2 fu5jh7} the canonical form in \eqref{MYS2 fu5} and the corresponding coefficients $A$ and $B$ in \eqref{MYS2 fu5} are
\begin{equation*}
\begin{split}
A&=\beta^2,\\
B&=2\frac{\alpha^4+2\kappa_1^2\alpha^2+\kappa_2^2}{\alpha^4-2\kappa_1^2\alpha^2+\kappa_2^2}\beta^2,
\end{split}
\end{equation*}
respectively. It is well-known that equation \eqref{MYS2 fu5} is parameterized by elliptic functions; see \cite{Baxter1982,RamaniCarsteaGrammaticosOhta2002} or \cite{HalburdKorhonen2007}. Here we incorporate the process of parametrization from \cite{HalburdKorhonen2007}. Define the parameters $k$ and $\varepsilon$ such that
\begin{equation}\label{MYS2 fu5jh1}
\begin{split}
A&=-\frac{1}{k\text{sn}^2\ \varepsilon},\\
B&=\frac{\text{cn}\ \varepsilon \ \text{dn}\ \varepsilon}{k\text{sn}^2\ \varepsilon},
\end{split}
\end{equation}
respectively. These choices of $A$ and $B$ imply that
\begin{equation}\label{MYS2 fu5jh2}
\begin{split}
k+k^{-1}=(B^2-A^2-1)A^{-1}.
\end{split}
\end{equation}
Therefore, considering equation \eqref{MYS2 fu5} as a quadratic equation for $\overline{f}$, and using the transformation $f=k^{1/2}\text{sn} \ u$, where $\text{sn}\ u$ denotes the Jacobi elliptic $\text{sn}$ function with argument $u$ and modulus $k$, we have
\begin{equation}\label{MYS2 fu5jh3}
\begin{split}
\text{sn}\ \overline{u}=\frac{\text{cn}\ \varepsilon \ \text{dn}\ \varepsilon \ \text{sn}\ u \pm \text{sn}
\ \varepsilon\  \text{cn}\ u \ \text{dn}\ u}{1-k^2\text{sn}^2\ \varepsilon\  \text{sn}^2\ u}.
\end{split}
\end{equation}
This is solved by $u=\varepsilon z+C$, where $C$ is a free parameter. Using the expressions of $A$ and $B$ in \eqref{MYS2 fu5jh1}, we rewrite equation \eqref{MYS2 fu5} as
\begin{equation}\label{MYS2 fu5jh3nh0}
\begin{split}
(\overline{f}-f)^2=(k\text{sn}^2\ \varepsilon)\overline{f}^2f^2+2\left(\text{cn}\ \varepsilon \ \text{dn}\ \varepsilon-1\right)\overline{f}f+k\text{sn}^2\ \varepsilon.
\end{split}
\end{equation}
By the above process, if we set
\begin{equation}\label{MYS2 fu5jh3nh}
\begin{split}
t=\varepsilon z,\quad f=k^{1/2}w(t,\varepsilon),
\end{split}
\end{equation}
then, since $f(z+1)=k^{1/2}w(t+\varepsilon,\varepsilon)$, by dividing $k\text{sn}^2\ \varepsilon$ on both sides of equation \eqref{MYS2 fu5jh3nh0} we get
\begin{equation}\label{MYS2 fu5jh4}
\begin{split}
&\frac{[w(t+\varepsilon,\varepsilon)-w(t,\varepsilon)]^2}{\text{sn}^2\ \varepsilon }\\
=&\, k^2w(t+\varepsilon,\varepsilon)^2w(t,\varepsilon)^2+\left(\frac{2\text{cn}\ \varepsilon\  \text{dn}\ \varepsilon-2}{\text{sn}^2\ \varepsilon}\right)w(t+\varepsilon,\varepsilon)w(t,\varepsilon)+1.
\end{split}
\end{equation}
Recall the Maclaurin series for $\text{sn}\ \varepsilon$, $\text{cn}\ \varepsilon$ and $\text{dn}\ \varepsilon$, respectively:
\begin{equation}\label{MYS2 fu5jh5}
\begin{split}
\text{sn}\ \varepsilon&=\varepsilon-(1+k^2)\frac{\varepsilon^3}{3!}+(1+14k^2+k^4)\frac{\varepsilon^5}{5!}+\cdots,\\
\text{cn}\ \varepsilon&=1-\frac{\varepsilon^2}{2!}+(1+k^4)\frac{\varepsilon^4}{4!}+\cdots,\\
\text{dn}\ \varepsilon&=1-k^2\frac{\varepsilon^2}{2!}+k^2(4+k^2)\frac{\varepsilon^4}{4!}+\cdots.
\end{split}
\end{equation}
By substituting the above series into \eqref{MYS2 fu5jh4} and then letting $\varepsilon\to0$, we obtain the following differential equation:
\begin{equation}\label{MYS2 fu5jh6}
\begin{split}
[w'(t,0)]^2=(k^2 w(t,0)^2-1)(w(t,0)^2-1),
\end{split}
\end{equation}
which is equation \eqref{MYS2 fu1}. In particular, for equation \eqref{MYS2 fu5jh7}, we see that this process yields $2\frac{\alpha^4+2\kappa_1^2\alpha^2+\kappa_2^2}{\alpha^4-2\kappa_1^2\alpha^2+\kappa_2^2}\to -1$ as $\varepsilon \to 0$. Recalling that $\alpha^4=\kappa_2^2$, this implies that $\kappa_2\to \pm\frac{\kappa_1}{3}$ as $\varepsilon \to 0$. By combining the results above together, we conclude that \eqref{yanagiharaeq14 co} has a continuum limit to the differential equation \eqref{MYS2}.

We now consider the five equations \eqref{intro_eq_list2_1}--\eqref{intro_eq_list2_4} and \eqref{intro_eq_list2_6}. We take the equation \eqref{intro_eq_list2_1} as an example to show that this equation can be transformed into the symmetric form, which is included in the QRT family, and have a continuum limit to \eqref{MYS2 fu1}. It is easy to see that solutions of equation \eqref{intro_eq_list2_1} also satisfy the following two equations:
\begin{equation}\label{yanagiharaeq11afujai1}
\begin{split}
\overline{f}^2+\eta^2&=\eta^2f^2,\\
\overline{f}^2-1&= \eta^2(f^2+\eta^2).
\end{split}
\end{equation}
Since all roots of $f(z)\pm i\eta=0$ have even multiplicities, we see that $\frac{f+i\eta}{f-i\eta}=h^2$ for some meromorphic function $h$. It follows that $f=i\eta\frac{h^2+1}{h^2-1}$. Denote $H=\frac{h^2+1}{2h}$. By dividing the first equation in \eqref{yanagiharaeq11afujai1} by the second equation in \eqref{yanagiharaeq11afujai1} on both sides, we obtain
\begin{equation}\label{yanagiharaeq11afujai2}
\begin{split}
\frac{\overline{f}^2+\eta^2}{\overline{f}^2-1}&=\frac{f^2}{f^2+\eta^2}=\left(\frac{h^2+1}{2h}\right)^2=H^2,
\end{split}
\end{equation}
i.e.,
\begin{equation}\label{yanagiharaeq11afujai3}
\begin{split}
\overline{f}^2H^2-(\overline{f}^2+H^2)-\eta^2=0,
\end{split}
\end{equation}
which is a biquadratic equation with respect to $\overline{f}$ and $H$. Instead of considering equation \eqref{yanagiharaeq11afujai3} directly, we may first re-scale $f$ and $H$ simultaneously by $f\to f/\kappa_1$ and $H\to H/\kappa_1$ with a constant $\kappa_1$ to obtain the equation $\overline{f}^2H^2-\kappa_1^2(\overline{f}^2+H^2)-\eta^2\kappa_1^4=0$, where $\kappa_1$ is a constant. Then we do the M\"obius transformations:
\begin{equation}\label{yanagiharaeq11afujai4}
\begin{split}
\overline{f}\to\alpha\frac{\overline{f}-\beta}{\overline{f}+\beta}, \quad H\to\alpha\frac{H-\beta}{H+\beta},
\end{split}
\end{equation}
with suitable constants $\alpha$ and $\beta$, and obtain the canonical form of the symmetric QRT map:
\begin{equation}\label{yanagiharaeq11afujai5}
\begin{split}
\overline{f}^2H^2+A(\overline{f}^2+H^2)+2B\overline{f}H+1=0,
\end{split}
\end{equation}
where $A$ and $B$ are both nonzero constants dependent on $\alpha$, $\beta$ and $\kappa_1$. The process of solving \eqref{MYS2 fu5} shows that equation \eqref{yanagiharaeq11afujai5} is parameterized by elliptic functions and $\overline{f}=H(\overline{\varphi})$ for an entire function $\varphi$. In fact, if we define the parameters $k$ and $\varepsilon$ as in \eqref{MYS2 fu5jh1} and consider equation \eqref{yanagiharaeq11afujai5} as a quadratic equation for $\overline{f}$ with respect to $H$, then using the transformation $H=k^{1/2}\text{sn} \ \varphi $ and $\overline{f}=k^{1/2}\text{sn}\ \overline{\varphi}$, where $\text{sn} \ \varphi$ denotes the Jacobi elliptic $\text{sn}$ function with argument $\varphi$ and modulus $k$, we have
\begin{equation}\label{yanagiharaeq11afujai8}
\begin{split}
\text{sn}\ \overline{\varphi}=\frac{\text{cn}\ \varepsilon \ \text{dn}\ \varepsilon \ \text{sn}\ \varphi \pm \text{sn}
\ \varepsilon\  \text{cn}\ \varphi \ \text{dn}\ \varphi}{1-k^2\text{sn}^2\ \varepsilon\  \text{sn}^2\ \varphi},
\end{split}
\end{equation}
which is solved by $\varphi=\varepsilon \phi+C$ such that $\varphi=\varphi(z)$ is an entire function satisfying $\varphi(z+1)=\varphi(z)+\varepsilon$, where $C$ is a free parameter. It follows that $\phi=\phi(z)$ is an entire function satisfying $\phi(z+1)=\phi(z)+1$. Thus $\phi(z)=\pi(z)+z$, where $\pi(z)$ is an arbitrary non-constant periodic function of period~1. We may suppose that $\pi(z)$ has a zero, say $\pi(z_0)=0$. Then $z_m=z_0+m$ is a zero of $\pi(z)$ for all integers $m\geq0$. It follows that the infinite sequence $\{z_m\}$ satisfies $z_m\to\infty$ as $m\to\infty$ and $\phi_m=\phi(z_m)=z_m$ for all $m$. Therefore, if we set
\begin{equation}\label{yanagiharaeq11afujai9}
\begin{split}
t=\varepsilon_m z_m,\quad H=k^{1/2}w(t,\varepsilon_m), \quad \overline{f}=k^{1/2}w(t+\varepsilon_m,\varepsilon_m),
\end{split}
\end{equation}
then we have from \eqref{yanagiharaeq11afujai5} that
\begin{equation}\label{yanagiharaeq11afujai10}
\begin{split}
&\frac{[w(t+\varepsilon_m,\varepsilon_m)-w(t,\varepsilon_m)]^2}{\text{sn}^2\ \varepsilon_m}\\
=&\, k^2w(t+\varepsilon_m,\varepsilon_m)^2w(t,\varepsilon_m)^2+\left(\frac{2\text{cn}\ \varepsilon_m \  \text{dn}\ \varepsilon_m-2}{\text{sn}^2\ \varepsilon_m}\right)w(t+\varepsilon_m,\varepsilon_m)w(t,\varepsilon_m)+1.
\end{split}
\end{equation}
For a fixed $t$, we choose $\varepsilon_m=\frac{t}{z_m}$. By using the Maclaurin series for $\text{sn}\, \varepsilon_m$, $\text{cn}\, \varepsilon_m$ and $\text{dn}\, \varepsilon_m$, respectively, in \eqref{MYS2 fu5jh5}, and then letting $\varepsilon_m\to0$, we obtain exactly the differential equation in \eqref{MYS2 fu5jh6}.
For each of the four equations \eqref{intro_eq_list2_2}, \eqref{intro_eq_list2_3} and \eqref{intro_eq_list2_4} and \eqref{intro_eq_list2_6}, by using the same method as above we may obtain an equation of the form in \eqref{yanagiharaeq11afujai3} with a certain meromorphic function $H$ and then also obtain the differential equation in \eqref{MYS2 fu5jh6} after taking a continuum limit. We omit those details.

\section{Concluding remarks}\label{Concluding remarks}

The Malmquist type difference equations \eqref{first_order_de_n} with $\deg_f(R(z,f))=n$ are revisited in this paper.
In section~\ref{An elementary derivation of the first 12 equations}, we first complete the classification for equation \eqref{first_order_de_n} with $\deg_f(R(z,f))=n$ by identifying one new equation \eqref{intro_eq_list2_6} left out in our previous work. We have actually derived the eleven equations \eqref{yanagiharaeq11 co}--\eqref{intro_eq_list2_6} using some recent observations on equation \eqref{first_order_de_n} in \cite{zhangkorhonen2022}. In section~\ref{Relations between Malmquist type differential and difference equations}, we study the relations between the Malmquist type differential and difference equations in the case $n=2$. The seven equations \eqref{lineareq}--\eqref{yanagiharaeq15 co} singled out from \eqref{first_order_de_n} have finite order meromorphic solutions and appear to be integrable from the viewpoint of the proposed difference analogue of the Painlev\'{e} property suggested by Ablowitz, Halburd and Herbst \cite{AblowitzHalburdHerbst2000}. We point out that each of the equations \eqref{lineareq}--\eqref{yanagiharaeq14 co} has a natural continuum limit to equations \eqref{riccati}, \eqref{MYS1} or \eqref{MYS2}. The process of taking a continuum limit from equation \eqref{yanagiharaeq14 co} to equation \eqref{MYS2} also applies to some more equations singled out from \eqref{first_order_de_n} in the case $n=2$, namely the five equations \eqref{intro_eq_list2_1}--\eqref{intro_eq_list2_4} and \eqref{intro_eq_list2_6}. These equations only have infinite order transcendental meromorphic solutions. However, they can also be mapped to the symmetric QRT map with respect to $\overline{f}$ and a meromorphic function $H$ dependent on $f$, so that $\overline{f}$ and $H$ are written in the form $\overline{f}=H(\varphi+\varepsilon)$ and $H=H(\varphi)$ with an argument $\varphi$ which is an entire function of $z$. In fact, by looking at the proof of the main theorems in \cite{zhangkorhonen2022} and the discussions in the last section of \cite{zhangkorhonen2022}, we see that most equations singled out from \eqref{first_order_de_n} with $n=2$ in the autonomous case in \cite{zhangkorhonen2022} can be written in the form
\begin{equation*}
\overline{f}^2+R_1^2=1,
\end{equation*}
or the form
\begin{equation*}\overline{f}^2R_2^2-(\overline{f}^2+R_2)+\kappa^2=0,
\end{equation*}
where $R_1$ and $R_2$ are rational functions in $f$ or in a certain meromorphic function $g$ such that $f=g^2-1$ or $f=a\frac{g^2-b}{g^2-c}$ for some constants $a,b,c$ such that $abc\not=0$, and thus are included in the QRT family defined in \eqref{QRTmapI}--\eqref{QRTmapIV}.

Recall from \cite{RamaniCarsteaGrammaticosOhta2002} that the QRT family defined in \eqref{QRTmapI}--\eqref{QRTmapIV} possesses an invariant which is biquadratic in $x_n$ and $y_n$:
\begin{equation}\label{MYS2 fu5qret}
\begin{split}
(\alpha_0&+K\alpha_1)x_n^2y_n^2+(\beta_0+K\beta_1)x_n^2y_n+(\gamma_0+K\gamma_1)x_n^2+(\delta_0+K\delta_1)x_ny_n^2\\
&+(\varepsilon_0+K\varepsilon_1)x_ny_n+(\zeta_0+K\zeta_1)x_n+(\kappa_0+K\kappa_1)y_n^2\\
&+(\lambda_0+K\lambda_1)y_n+(\mu_0+K\mu_1)=0,
\end{split}
\end{equation}
where $K$ plays the role of the integration constant. In the symmetric case, the invariant in \eqref{MYS2 fu5qret} becomes just \eqref{MYS2 fu5qret1}. In the generic case, it is shown in \cite{RamaniCarsteaGrammaticosOhta2002} that, by doing two M\"obius transformations $x_n\to\frac{\alpha_1 x_n+\alpha_2}{\alpha_3x_n+\alpha_4}$ and $y_n\to\frac{\beta_1 y_n+\beta_2}{\beta_3y_n+\beta_4}$ with suitable constants $\alpha_i$ and $\beta_j$, respectively, \eqref{MYS2 fu5qret} can also be mapped into the symmetric form:
\begin{equation}\label{MYS2 fu5qretgf}
\begin{split}
x_n^2y_n^2+A(x_n^2+y_n^2)+2Bx_ny_n+1=0,
\end{split}
\end{equation}
where $A$ and $B$ are constants. The process of solving \eqref{yanagiharaeq11afujai5} in section~\ref{Relations between Malmquist type differential and difference equations} shows that $x_n$ and $y_n$ in \eqref{MYS2 fu5qretgf} are parameterized by elliptic functions and $x_n=y_n(\overline{\varphi})$ with some entire function $\varphi$. Combining this fact and the process of taking continuum limit from equations \eqref{MYS2 fu5} and \eqref{yanagiharaeq11afujai5} to the differential equation \eqref{MYS2} in section~\ref{Relations between Malmquist type differential and difference equations}, we conclude that the QRT family defined in \eqref{QRTmapI} and \eqref{QRTmapII} always has a continuum limit to the first order differential equation \eqref{MYS2} in the generic case.

\end{document}